\documentclass[reqno]{amsart}
\usepackage{amsfonts,amsmath,amssymb,epsfig,mathrsfs,mathtools,esint}

\newtheorem{theorem}{Theorem}[section]
\newtheorem{lemma}[theorem]{Lemma}
\newtheorem{proposition}[theorem]{Proposition}
\newtheorem{corollary}[theorem]{Corollary}

\theoremstyle{definition}
\newtheorem{definition}[theorem]{Definition}

\theoremstyle{remark}
\newtheorem{remark}[theorem]{Remark}

\numberwithin{equation}{section}
\newcommand{\eps}{{\varepsilon}}
\newcommand{\de}{\partial}
\newcommand{\ph}{\varphi}
\newcommand\R{{\mathbb R}}
\newcommand\Q{{\mathbb Q}}
\newcommand\N{{\mathbb N}}

\newcommand\Pe{{\rm Per}\,}

\newcommand\sdif{{\vartriangle}}
\newcommand\dist{{\rm {dist}}}
\newcommand\dd{{\rm {d}}}

\renewcommand\div{{\rm {div}}}

\newcommand\M{\mathbf{M}}

\newcommand{\cO}{{\mathcal{O}}}
\newcommand{\cF}{{\mathcal{F}}}
\newcommand{\cH}{{\mathcal{H}}}

\newcommand{\cG}{{\mathcal{G}}}

\newcommand{\cE}{{\mathcal{E}}}

\makeatletter
\def\blfootnote{\xdef\@thefnmark{}\@footnotetext}
\makeatother

\title{Mean-convex sets and minimal barriers}
\author[E.~Spadaro]{Emanuele Spadaro}
\address{Max-Planck-Institut f\"ur Mathematik in den Naturwissenschaften, Leipzig}
\email{spadaro@mis.mpg.de}

\begin{document}

\begin{abstract}
A mean-convex set can be regarded as a barrier for the construction of minimal surfaces. Namely, if $\Omega \subset \R^3$ is mean-convex and $\Gamma \subset \de\Omega$ is a null-homotopic (in $\Omega$) Jordan curve, then there exists an embedded minimal disk $\Sigma \subset \bar\Omega$ with boundary $\Gamma$. Does a mean-convex set $\Omega$ contain all minimal disks with boundary on $\de\Omega$? Does it contain the solutions of Plateau's problem?
We answer this question negatively and characterize the least barrier enclosing all the minimal hypersurfaces with boundary on a given set.
\end{abstract}

\maketitle

\section{Introduction}
\blfootnote{The author is grateful to the Mathematisches Forschungsinstitut Oberwolfach, for the present work originated from the Arbeitsgemeinschaft on Minimal Surfaces held in the Fall 2009.}
A mean-convex set $\Omega$ in a Riemannian manifold is a local \textit{barrier} for minimal hypersurfaces, for it satisfies a well-known strong maximum principle:
if a cycle $\Gamma\subset\Omega$ can be parametrized as a graph, then the minimal hypersurface $\Sigma$ with boundary $\Gamma$ is contained in $\Omega$, and $\Sigma \cap \de\Omega \neq \emptyset$ if and only if $\Sigma \subset \de \Omega$.
More interesting, mean-convex sets can also be regarded as barriers for the construction of minimal surfaces. Indeed, following the work by Meeks and Yau \cite{MeYa}, given a mean-convex set $\Omega$ in a homogeneous $3$-dimensional Riemannian manifold $N$ and given a Jordan curve $\Gamma\subset\de \Omega$ which is null-homotopic in $\Omega$, there exists an embedded minimal disk $\Sigma$ such that $\de\Sigma=\Gamma$ and $\Sigma\subset\bar\Omega$ (see also \cite{AlSi} for convex sets $\Omega$).

However, a mean-convex set $\Omega$ may fail to be a \textit{global barrier}. There are simple examples for this phenomenon due to topological obstructions (for instance, the case of a boundary Jordan curve $\Gamma$ which is not null-homotopic in $\Omega$). Nevertheless, as we show in \S~\ref{s.mc barrier}, this global barrier principle may also fail in the simplest case of $\Omega \subset \R^3$ homeomorphic to a ball and $\Gamma \subset \Omega$ a Jordan curve, as well as if we restrict to area minimizing disks.

This arises the question: which is the least global barrier for all minimal hypersurfaces with boundary in $\Omega$?
A set $\Theta\subset\R^n$ is called a \textit{global barrier} if:
\[
 \Sigma\;\text{minimal hypersurface, }\; \de\Sigma \subset \Theta \quad \Longrightarrow \quad \Sigma \subset \Theta.
\]

In this paper we address the issue of characterizing the minimal barrier containing a set $\Omega$, here called the \textit{mean-convex hull} of $\Omega$:
\begin{equation}\label{e.def mc}
 \Omega^{mc} := \bigcap_{\Omega\subset\Theta \in \mathcal{A}} \Theta,
\end{equation}
where $\mathcal{A}$ denotes the family of global barriers in $\R^n$.
Few remarks are in order.
\begin{enumerate}
 \item Clearly, the closed convex hull $\Omega^{co}$ is a global barrier containing $\Omega$, hence the intersection in \eqref{e.def mc} is non-trivial. Nevertheless, $\Omega^{co}$ may not be the smallest one (see the examples in \S~\ref{s.mc barrier}).
 \item According to what said before, unless for the usual hulls, a mean-convex set does not need to coincide with its mean-convex hull.
 \item Nevertheless, by the definition $\Omega^{mc}$ turns out to be a hull, i.e.~$(\Omega^{mc})^{mc}=\Omega^{mc}$.
\end{enumerate}

Similar notions of mean-convex hull have been introduced for minimal hypersurfaces spanning a fixed extreme boundary, see \cite{White, Coskunuzer}.
The above defined mean-convex hull $\Omega^{mc}$ has in principle no topological structure and enjoys no regularity. The main result of the paper is to prove that, in dimension $n\leq 7$, the mean-convex set $\Omega^{mc}$ has actually a regular (optimal $C^{1,1}$) boundary and is in fact a mean-convex set.

\begin{theorem}\label{t.main}
Let $\Omega \subset \R^n$, with $n\leq 7$, be a bounded closed set with $\de\Omega \in C^{1,1}$. Then, $\Omega^{mc}$ is a closed mean-convex set with $C^{1,1}$ regular boundary. Moreover, $\de \Omega^{mc} \setminus \Omega$ is a minimal hypersurface with boundary in $\Omega$.
\end{theorem}

In particular, if $n=3$ and $\Omega$ is connected, by the property (3) above and Theorem~\ref{t.main} it follows that the least barrier for minimal hypersurfaces with boundary in $\Omega$ is actually a homology ball enclosing the set $\Omega$ whose boundary either touches $\Omega$ or is a minimal surface.

\subsection{Heuristics of the proof}
Given the global nature of $\Omega^{mc}$ (in particular the lack of a priori information concerning its topology and regularity), a purely partial differential equation approach does not seem to be tailored to distinguish between the local and global barrier property. Similarly, the solutions to several variational problems which can be naturally associated to the mean-convex hull, such as, e.g., the minimizing hulls considered by Ilmanen and Huisken \cite{HuIl}, do not lead in general to a global barrier.

The main idea of the paper is to use an evolution approach. Roughly speaking, minimal surfaces with boundary in $\Omega$ can be seen as stationary solutions of the mean curvature flow with fixed boundary in $\Omega$, interpreted as an obstacle to the flow. Hence, remaining within this intuition, one could try to characterize the mean-convex hull in terms of the asymptotic evolution of a mean curvature flow with obstacle. To this purpose, we consider the evolution of the boundaries of sets $F_t$ containing $\Omega$ such that the normal velocity $\vec v_{F_t}$ at any point of $\de F_t$ satisfies the equation:
\begin{equation}\label{e.eq flow}
\vec v_{F_t}(x)=
\begin{cases}
\vec H_{\de F_t}(x) & \text{if } x\in \de F_t\setminus\Omega,\\
{\max}\left\{\vec H_{\de F_t}\cdot \vec n_{F_t},0\right\}\,\vec n_{F_t} & \text{if } x\in \de F_t\cap\Omega,
\end{cases}
\end{equation}
where $\vec n_{F_t}$ denotes the unit external normal to $\de F_t$.
In words, the evolution of $F_t$ follows the classical mean curvature flow equation away from the obstacle $\Omega$ while on the boundary of $\Omega$ satisfies a unilateral constraint, namely it can leave the obstacle if its mean curvature vector points outward, otherwise it stops.\footnote{After the paper was completed, we learned that a similar evolution has been considered in a recent preprint by Almeida, Chambolle and Novaga \cite{AlChNo}.} The idea is to show that
\[
\Omega^{mc}=\lim_{t \to +\infty} F_t.
\]

Clearly, such heuristic approach cannot naively work in a general framework. As it is well known, the mean curvature flow develops singularities in finite time, thus not allowing a pointwise meaning to \eqref{e.eq flow} -- thought possible under specific geometric assumptions. Hence, in order to define such a flow we need to consider a generalized flow with obstacle. There are by now many approaches to weak mean curvature flow: Brakke's varifolds flow \cite{Br}, the partial differential equations approach by Evans and Spruck \cite{EvSp1,EvSp2,EvSp3,EvSp4} and Chen, Giga and Goto \cite{CGG}, the elliptic regularization by Ilmanen \cite{Il3}, and the barrier approach developed by Ilmanen \cite{Il1,Il2}, De Giorgi \cite{DG1,DG2}, Bellettini and Novaga \cite{BeNo}, White \cite{Wh1}. In this paper we generalize to the case of obstacle the mean curvature flow of Caccioppoli sets introduced by Almgren, Taylor and Wang \cite{ATW} (see also Luckhaus and Sturzenhecker \cite{LuSt}).
However, since in some special cases the different approaches turn out to be closely related, we do believe that similar arguments to ours may also be apply within different choices for the weak flow.

\subsection{Overview}
The proof of Theorem~\ref{t.main} is made in different steps and diverts from the heuristic sketch given above because of several technical issues, mainly due to the lack of regularity for the weak flow. In particular, the restriction on the space dimension $n\leq 7$ in Theorem~\ref{t.main} is due to the use of curvature estimates for stable minimal surfaces needed in the proof of the regularity of the mean-convex hull. However, similar (though weaker) partial regularity results can be obtained in higher dimension.

The paper is organized as follows. In \S~\ref{s.mc barrier} we give the main definitions, fix the notation and illustrate some counterexamples to the equivalence between mean-convexity and the notion of barrier.
Then, after recalling the basic notion of geometric measure theory (which, although essential for our arguments, we keep to the minimum), we develop in \S~\ref{s.flow} a weak theory of mean curvature flow with obstacle after \cite{ATW,LuSt} (since not needed for the main result, the proof of the existence of a limiting weak flow is postponed to the Appendix~\ref{a.estimates}). In \S~\ref{s.mc flow} we specialize our arguments to the case of monotone flows starting from a minimizing hull. 
In particular, we will show that in this case one can define uniquely a maximal solution to the flow, which has an asymptotic limit with uniform curvature bounds. Finally, in \S~\ref{s.mc hull} we prove the main results in Theorem~\ref{t.main}.

\subsection{Other ambient manifolds} All the results of the paper hold unchanged if one replaces $\R^n$ with an arbitrary Riemannian manifold.
The proofs are, indeed, simple modifications of the ones given in $\R^n$.
Note, however, that the implication of Theorem~\ref{t.main} on the topology of the mean-convex hull in dimension $n=3$ may fail to be true.

\section{Mean-convex sets and barriers}\label{s.mc barrier}
Throughout the next sections, $\Omega$ denotes a bounded closed set in $\R^{n}$ with $C^{1,1}$ boundary $\de \Omega$. We let $\nu$ be the external unit normal to $\de\Omega$ and $\vec H_{\de\Omega}$ the mean curvature vector of $\de\Omega$.

One says that $\Omega$ is mean-convex if $\vec H_{\de\Omega}$ is pointing inside $\Omega$ at every point, i.e.~$\vec H_{\de\Omega}\cdot \nu\leq0$. As pointed out in the Introduction, mean-convex sets are local barrier to minimal hypersurfaces because of the strong Hopf maximum principle (see, e.g., \cite[Section~1.7]{CoMi}). Moreover, following the work by Meeks and Yau \cite{MeYa}, a mean-convex set can be used as a global barrier for the construction of minimal surfaces (their result holds in fairly more general hypotheses on the ambient manifold and on the regularity of the mean-convex set, which may be assumed piecewise $C^{1,1}$). Here the term disk refers to a smooth $2$-dimensional surface with boundary, having the topology of the planar disk $D=\{(x,y) \in \R^2 : x^2+y^2 \leq 1\}$.

\begin{theorem}[Meeks \& Yau \cite{MeYa}]\label{t.meeks--yau}
Let $\Omega\subseteq\R^3$ be a bounded mean-convex set and $\Gamma\subseteq\de\Omega$ a closed curve, null-homotopic in $\Omega$ (i.e.~there exists a disk contained in $\Omega$ with boundary $\Gamma$).
Then, there exists an embedded minimal disk $\Sigma\subseteq\Omega$ such that $\de\Sigma=\Gamma$ and $\Sigma$ minimizes the area among all the disks in $\Omega$ with the same boundary.
\end{theorem}

More precisely, the theorem asserts that every solution $\Sigma$ (actually, it may not be unique) of the \textit{constrained} Plateau problem, namely minimizing the area among all disks contained in $\Omega$ with the same boundary, satisfies the minimal surface equation $\vec H_\Sigma\equiv 0$, i.e.~is a stationary point for the \textit{unconstrained} area functional. In particular, Theorem~\ref{t.meeks--yau} does not imply that the Douglas--Rado solution of the Plateau problem with boundary $\Gamma$ is contained in $\Omega$.

There are several examples of mean-convex sets $\Omega$ and cycles $S \subset \de\Omega$ such that the solution of the Plateau problem is not contained in $\Omega$. The simplest ones are due to topological obstructions. For instance, this is the case when the curve $\Gamma$ in Theorem~\ref{t.meeks--yau} is not null-homotopic in $\Omega$. For example, if $\Omega$ is a rotationally symmetric mean-convex torus and $\Gamma$ any parallel circle on its boundary, the unique minimizing surface with this boundary is the flat disk, which is not contained inside the torus. (Note that it follows from this considerations that the mean-convex hull of the torus coincides with its convex hull, thus showing that $C^{1,1}$ is the optimal regularity.)

Similarly, there are simple examples in the case of not connected boundaries. Consider, for instance, two parallel circles in the boundary of a dumb-bell mean-convex set $\Omega$: choosing appropriately the ratio between the radius of the circles and the distance between them, it can be proved that the minimizing surface is actually the catenoid which partially bends outside $\Omega$ (see Figure~\ref{f.cat} for a self-explanatory picture: the details are left to the reader).

\begin{figure}[htp]\label{f.cat}
\centering
\includegraphics[height=2.5cm]{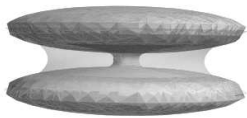}
\caption{Catenoid partially bending outside a mean-convex set.}
\end{figure}

\subsection{Counterexample in the hypothesis of Meeks \& Yau's theorem}
However, it is also possible to find a counterexample under the hypothesis of Theorem~\ref{t.meeks--yau}, namely when $\Gamma$ is a null-homotopic simple Jordan curve. Moreover, in order to avoid any topological obstruction, we will also allow $\Omega$ to be a mean-convex set homeomorphic to the $3$-dimensional ball.

Since the existence of such example was not known to the author, we give in this section a fairly detailed description using some results is Geometric Measure Theory.

\begin{figure}[htp]
\centering
\includegraphics[height=2.5cm]{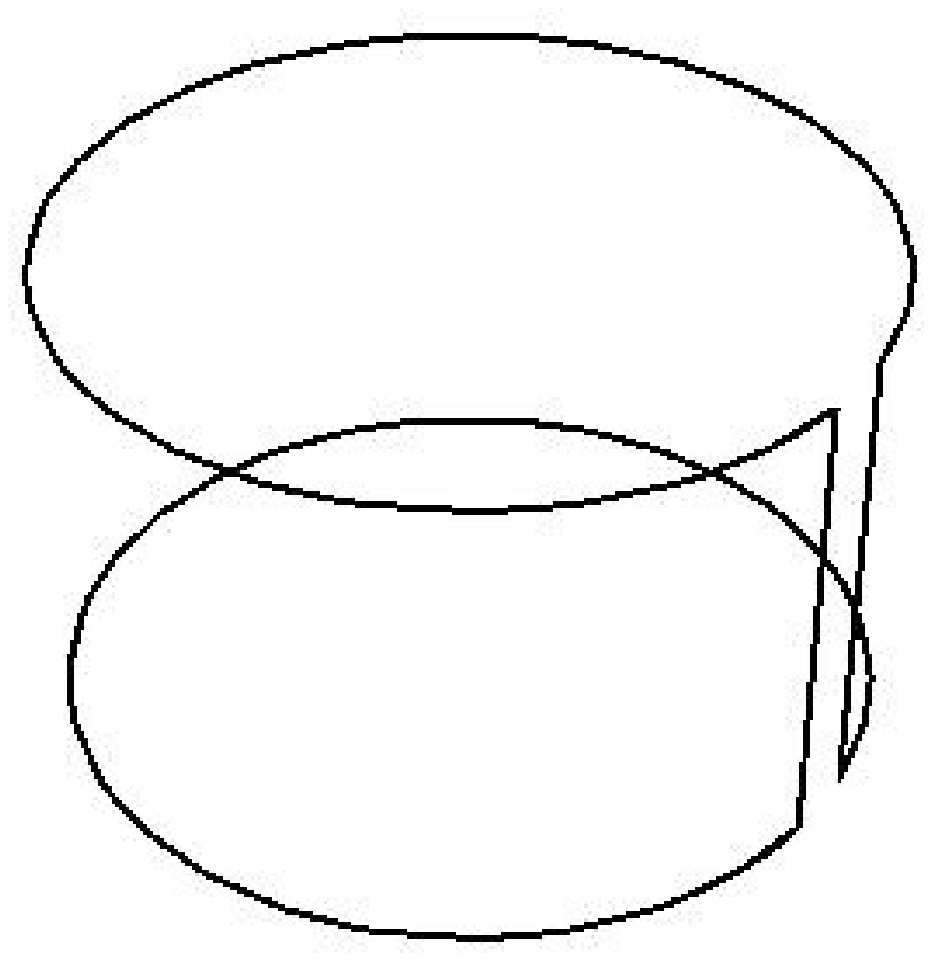}
\hspace{1cm}
\includegraphics[height=2.5cm]{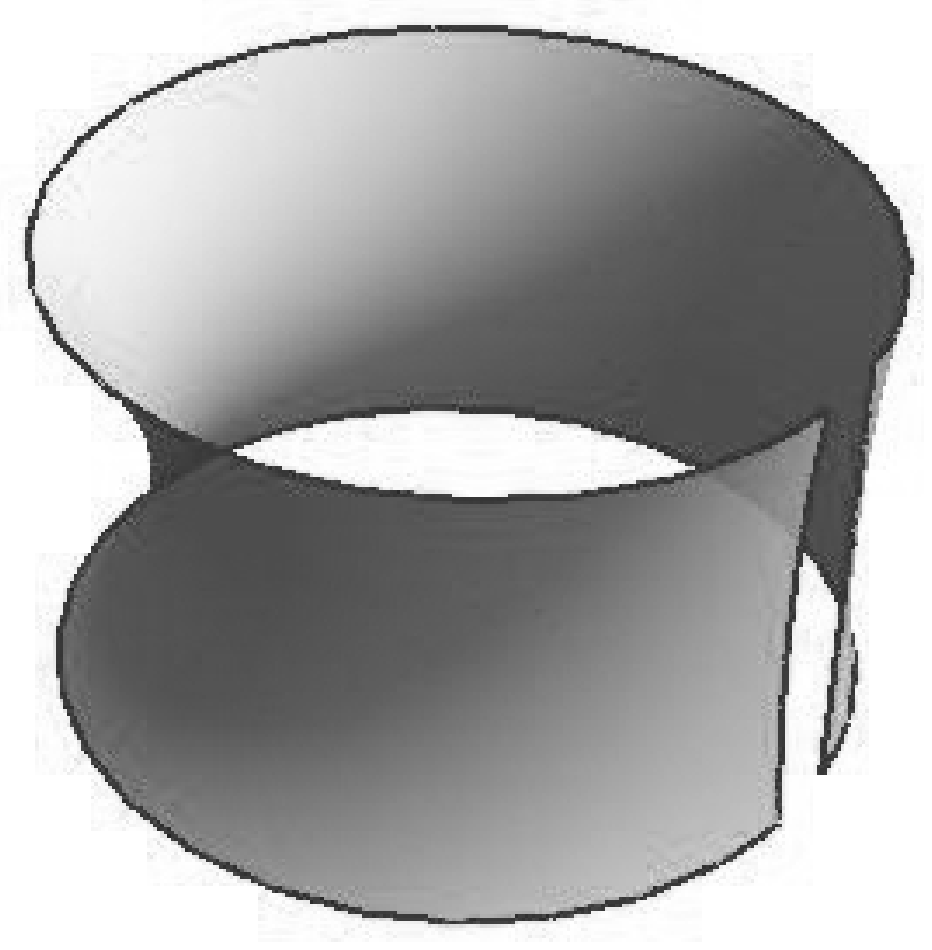}
\hspace{1cm}
\includegraphics[height=2.5cm]{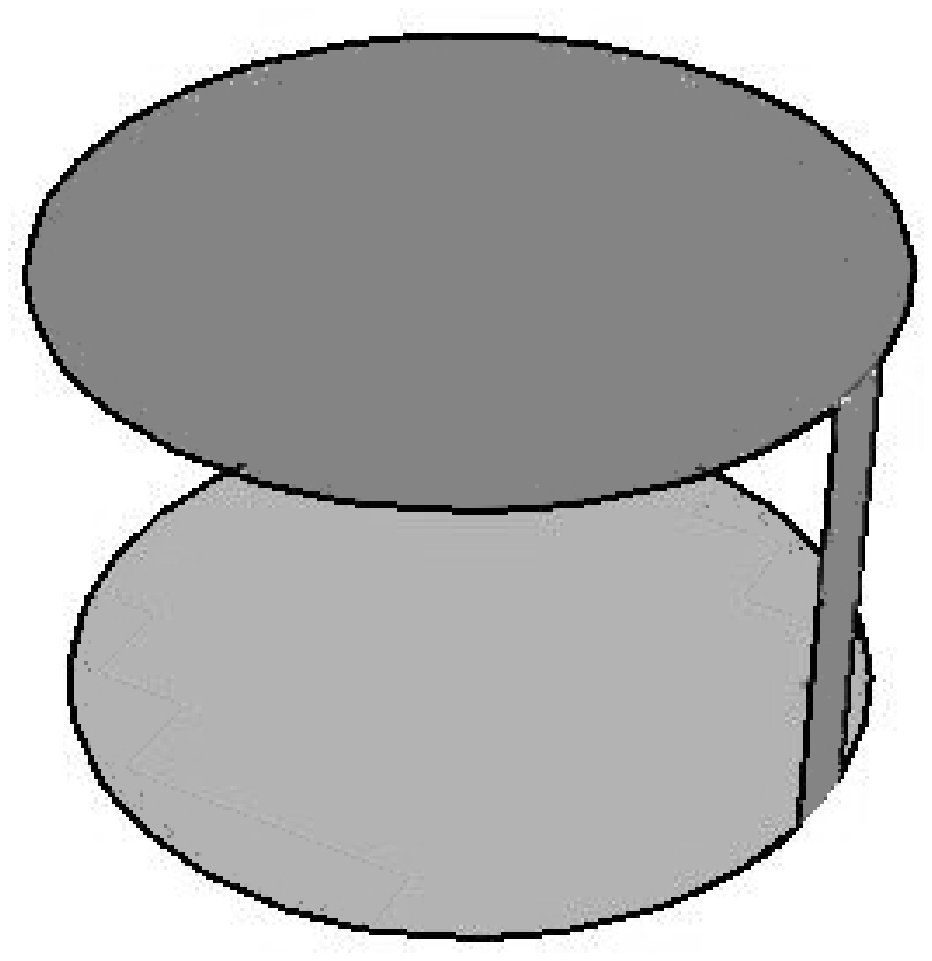}
\caption{Curve bounding at least two minimal disks (approximate solutions drawn).}
\label{f.Gamma}
\end{figure}

Our starting point is the well-known example of a Jordan curve bounding at least two different minimal disks (see, for example, \cite[\S~389]{Ni}).
Let us fix cylindrical coordinates in $\R^3$:
\[
 (x,y,z)=(\rho \cos\theta,  \rho \sin\theta, z) \quad \text{with } \; (\theta,\rho,z)\in[0,2\pi)\times[0,+\infty)\times\R.
\]
For $\theta_0>0$ a parameter to be fixed momentarily, let $\Omega_{\theta_0}$ be the following closed set (see Figure~\ref{f.Omega} for two views of this domain):
\[
\Omega_{\theta_0}:=\Big\{(\theta,\rho,z)\,:\theta_0\leq\theta\leq2\,\pi,\;|z|\leq L,\;a\cosh (z/a)\leq\rho\leq1\Big\},
\]
where $L>h:=0.6$ and $0<a<1$ are fixed in such a way that $a\cosh (L/a)<1$.
Note that such a choice of parameters is possible, for example $L=0.62$ and $a=0.5$.
Let $\Gamma\subseteq\de\Omega$ be the curve given by (see Figure~\ref{f.Gamma} left):
\[
\Gamma_{\theta_0}:=\Big\{(\theta,1,z)\,:(\theta,z)\in\de\big([2\,\theta_0,2\,\pi]\times[-h,h]\big)\Big\}.
\]

\begin{figure}[htp]
\centering
\includegraphics[height=3cm]{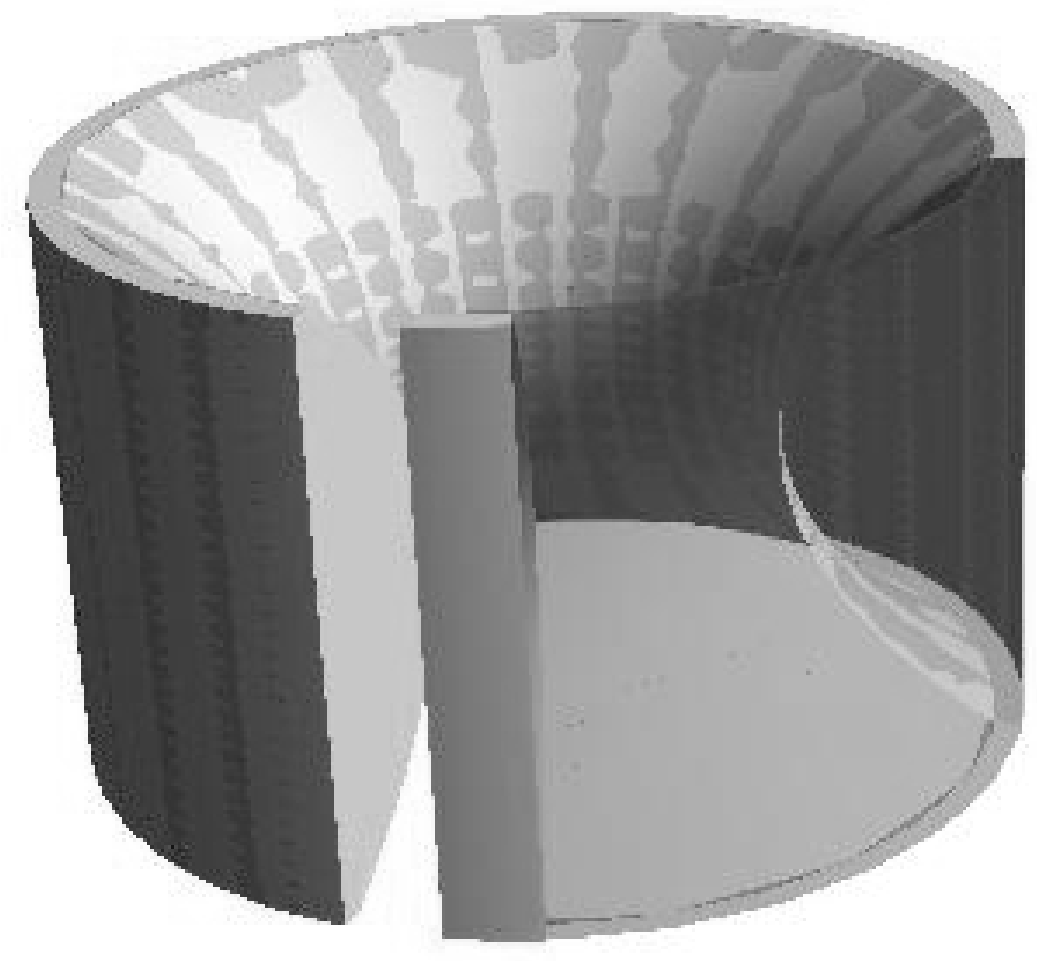}
\hspace{1cm}
\includegraphics[height=3cm]{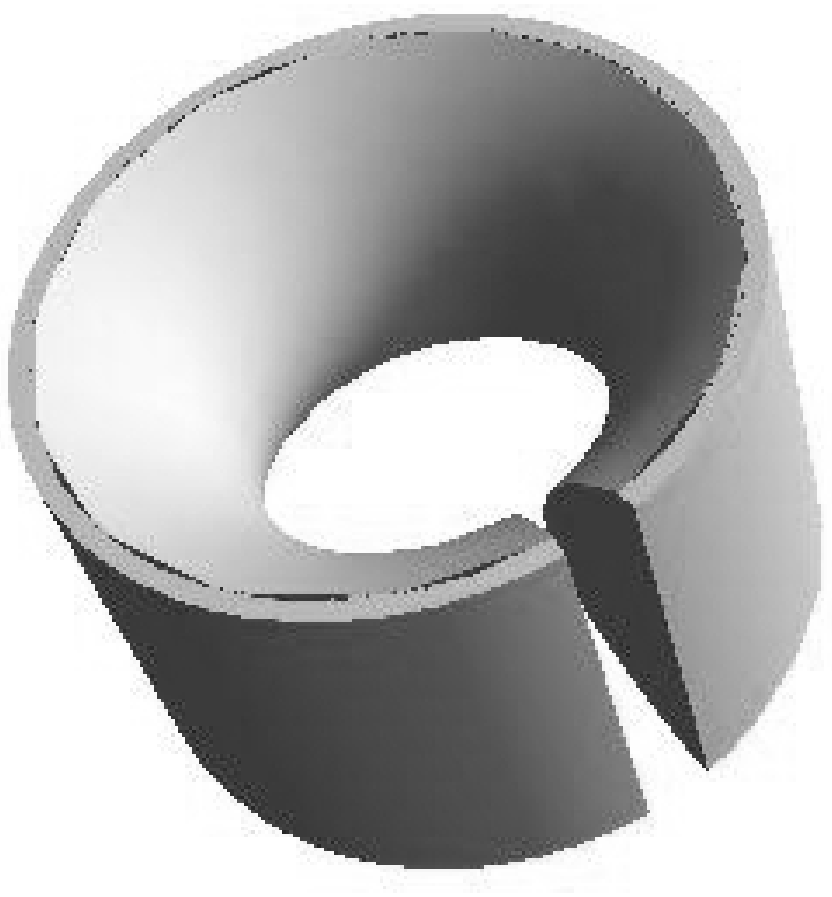}
\caption{Two views of the mean-convex domain $\Omega_{\theta_0}$ (on the left part the transparency shows the interior section).}
\label{f.Omega}
\end{figure}

It is well-known that the area minimizing surface with boundary two axial unitary circles on parallel planes distant $2\,h$ is the union of the two disks 
\begin{gather}
D_+ := \big\{(\theta, \rho, h)\; : \; \theta \in [0,2\pi), \; \rho \in [0,1] \big\},\\
D_{-} := \big\{(\theta, \rho, -h)\; : \; \theta \in [0,2\pi), \; \rho \in [0,1] \big\}
\end{gather}
(see, for example, \cite[\S~389]{Ni} and \cite{Sh}).

By compactness of integral currents, the minimizers $\Sigma_{\theta_0}$ of the area with boundary $\Gamma_{\theta_0}$ converge as $\theta_0\to0$ to a current $\Sigma$ with boundary the two circles $\de D_+$, $\de D_{-}$, and
\begin{equation}\label{e.lsc}
 \M(\Sigma) \leq \liminf_{\theta_0 \to 0} \M(\Sigma_\theta),
\end{equation}
(here $\M$ stands for the \textit{mass} of a current, that is the analog of the volume measure in Geometric Measure Theory).
It is a consequence of the Bridge Principle for minimal surfaces \cite[Theorem 2.2]{White94} that $\Sigma = D_+ \cup D_{-}$.
Indeed, if this is not the case, then being the two disks the absolute minimizers,
\begin{equation}\label{e.disks}
 \M (D_+ \cup D_{-}) < \M(\Sigma) \stackrel{\eqref{e.lsc}}{\leq} \liminf_{\theta_0 \to 0} \M(\Sigma_\theta).
\end{equation}
By the Bridge Principle, for every $\eps>0$ there exists $\theta_\eps>0$ and an integer rectifiable current $T_{\eps}$ such that $\de T_{\eps} = \Gamma_{\theta_\eps}$ and
\[
 \M(T_\eps) \leq \M (D_+ \cup D_{-}) + \eps,
\]
which together with \eqref{e.disks} contrasts the minimizing property of $\Sigma_{\theta_\eps}$ for $\eps$ sufficiently small.

By a simple consequence of the regularity theory for minimal surfaces, this convergence is smooth away from the points 
\[
(\theta,\rho,z) = (0,1,\pm\,h),
\]
and $\Sigma_{\theta_0}$ is contained in a neighborhood of
\[
D_+\cup D_{-} \cup \big\{(0,1,z) : |z|\leq h \big\},
\]
for $\theta_0$ sufficiently small.
In particular, for $\theta_0$ small enough, the minimizing disk with boundary $\Gamma_{\theta_0}$ resembles the surface in Figure~\ref{f.Gamma} on the right, and therefore is not contained in $\Omega_{\theta_0}$. Both $\Omega_{\theta_0}$ and $\Gamma_{\theta_0}$ are not smooth, but piecewise smooth. Nevertheless, since all the angles between the faces of $\Omega_{\theta_0}$ are less than $\pi$, it is not difficult (though boring) to modify the above example and reduce to a smooth mean-convex domain and a smooth Jordan curve.

\subsection{Mean-convex hull $\neq$ convex hull}
It follows directly from the definition that in the plane the mean-convex hull coincides with the convex hull.
Nevertheless, a simple example shows that the two notions do not need to coincide in dimension $n\geq 3$. Consider the set contained between a vertical catenoid and two horizontal parallel planes, i.e.
\[
\Omega=\Big\{(x,y,z)\::\:|z|\leq1,\;x^2+y^2\leq \cosh(z)^2\Big\}\subset\R^3,
\]
(the fact that $\Omega$ is not smooth is not essential, for the example can be modified accordingly). Clearly,
\[
\Omega^{co} = \{x^2+y^2\leq\cosh(1)^2\}. 
\]

Nevertheless, it is not difficult to show that $\Omega=\Omega^{mc}$. To see this, let $\Sigma$ be a minimal hypersurface with $\de\Sigma \subset \Omega$.
By the convex hull property, every minimal surface with boundary in $\Omega$ is contained in $\Omega^{co}$, hence, in particular, $\Sigma\subseteq\{|z|\leq 1\}$. On the other hand, consider the foliation by rescaled catenoids:
\[
\{|z|\leq 1\}\setminus \Omega=\bigcup_{\lambda\geq1}\textup{Cat}_\lambda,
\]
where
\[
\textup{Cat}_{\lambda}:=\Big\{(x,y,z)\::\:|z|\leq1,\;x^2+y^2=\lambda^2\, \cosh(z/\lambda)^2\Big\}.
\]

\begin{figure}[ht]
\centering
\includegraphics[height=2.0cm]{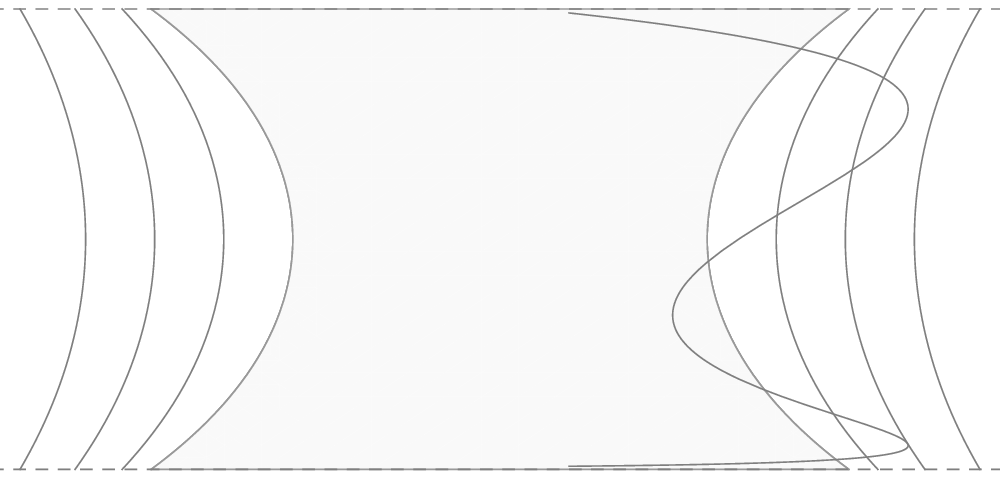}
\caption{Catenoids' foliation.}
\label{f.fol}
\end{figure}

Let $\lambda_{{\max}}$ the maximum $\lambda$ such that $\Sigma\cap \textup{Cat}_\lambda\neq\emptyset$ and assume $\lambda_{{\max}}>1$, i.e.~$\Sigma$ is not contained in $\Omega$.
By the strong maximum principle, it follows that $\Sigma\equiv\textup{Cat}_{\lambda_{{\max}}}$, thus contradicting $\de\Sigma\subset\Omega$ and implying that $\Sigma\subset\Omega$, i.e.~$\Omega=\Omega^{mc}$.

\section{Mean curvature flow with obstacle}\label{s.flow}
In this section we develop a weak mean curvature flow of Caccioppoli sets with obstacle.
We follow closely the approach of Almgren, Taylor and Wang \cite{ATW}, as revisited by Luckhaus and Sturzenhecker \cite{LuSt}. This is done in two steps, first introducing a discrete in time approximation of the flow; then, passing into the limit in the time step.

We start recalling the few notions of Geometric Measure Theory which are needed in the sequel (more details on Caccioppoli sets can be found in the monograph \cite{Gi}).

\subsection{Caccioppoli sets}\label{ss.caccioppoli}
A measurable set $E\subset\R^n$ is said to be a Caccioppoli set or a set of finite perimeter if there exist sets $E_j\subset\R^n$ with smooth boundary $\de E_j\in C^1$ such that $\chi_{E_j}\to\chi_E$ in $L^1(\R^n)$ and
\[
\liminf_{j\to+\infty}\cH^{n-1}(\de E_j)<+\infty.
\]
Here, as usual, $\cH^{n-1}$ denotes the $(n-1)$-dimensional Hausdorff measure and $\chi_E$ the characteristic function of the set $E$, namely
\[
\chi_E(x)=
\begin{cases}
1 &\text{if }\;x\in E,\\
0 &\text{if }\;x\notin E.
\end{cases}
\]
Note that, according to the above definition, a Caccioppoli set is defined up to a set of Lebesgue measure zero, for $\chi_E\in L^1$ identifies an equivalent class of measurable functions. Nevertheless, we will always assume to have fixed a pointwise representative of $E$ which satisfies the following condition:
\[
x\in\de E\quad\Longleftrightarrow\quad0<|B_r(x)\cap E|<\omega_n\,r^n \quad\forall\;r>0,
\]
where $|A|$ denotes the Lebesgue measure of a measurable set $A \subseteq \R^n$.

The measure of the boundary of $E$ in a open set $\cO\subset\R^n$, also called the perimeter of $E$ in $\cO$, is then given by the minimum limit of the measure in $\cO$ of the boundaries of the approximating sets, i.e.
\[
\Pe(E,\cO):=\inf\left\{\liminf_{j\to+\infty}\cH^{n-1}(\de E_j\cap\cO)\,:\,\chi_{E_j}\to\chi_E\;\;\text{in }\;L^1(\cO),\;\de E_j\in C^1\right\}.
\]

We will often write $\Pe(E)$ for $\Pe(E,\R^n)$.
Moreover, it turns out that, in case $\de E\in C^1$, then $\Pe(E,\cO)=\cH^{n-1}(\de E\cap\cO)$, thus justifying the term ``perimeter''.
An easy consequence of the definition (by choosing appropriate transversal approximations -- see also \cite[Proposition~3.38]{AFP}) is the inequality:
\begin{equation}\label{e.union int}
\Pe(E\cup F,\cO)+\Pe(E\cap F,\cO)\leq \Pe(E,\cO)+\Pe(F,\cO).
\end{equation}

Finally, we will use often the following two properties of Caccioppoli sets.
\begin{enumerate}
\item \textit{Lower semicontinuity}:
\[
\Pe(E,\cO)\leq \liminf_{j\to+\infty}\Pe(E_j,\cO),\quad\forall\;\chi_{E_j}\to\chi_E\;\;\text{in }\;L^1(\cO).
\]
\item \textit{Compactness}: given $E_j\subseteq B_R\subset\R^n$ with $\sup_j\Pe(E_j)<+\infty$, there exists $E\subset\R^n$ and a subsequence $(E_{j_k})_{k\in\N}$ such that
\[
\chi_{E_{j_k}}\to\chi_E\quad\text{in}\quad L^1(\R^n),\quad \text{as }\; k\to+\infty.
\]
\end{enumerate}

\subsection{Discrete in time approximate flow with obstacle}\label{ss.discrete}
In what follows $\Omega \subset \R^n$ is a closed bounded set with $C^{1,1}$ boundary and $E_0\subseteq\R^n$ is the initial bounded closed set of the evolution such that
\[
|E_0|=0 \quad \text{and} \quad \Omega\subset E_0.  
\]
We define the approximate flow of time step $h>0$ in the following way (for the heuristics motivating this definition we refer to the arguments for the unconstrained flow in \cite{ATW}).
We set $E^{(h)}_0:=E_0$ and, given $E^{(h)}_i$ for some $i\in\N$, we let $E^{(h)}_{i+1}$ be a minimizer of the functional $\cF(\cdot, h, E^{(h)}_{i})$ given by
\begin{equation*}
\cF(E,h, E^{(h)}_{i}) := \Pe(E)+\int_{E\sdif E^{(h)}_{i}}\frac{\dist\big(x,\de E^{(h)}_{i}\big)}{h}\;\dd x,
\end{equation*}
where the minimum is taken among all the sets $E$ containing $\Omega$ a.e.,
\[
\cF(E^{(h)}_{i+i},h, E^{(h)}_{i}) = \min \left\{\cF_{E^{(h)}_{i}}(E)\,:\,E\supset\Omega\;\;\text{a.e.}\right\}.
\]

It is clear that, thanks to the compactness and the semicontinuity properties (1) and (2) \S~\ref{ss.caccioppoli}, this minimum problem is well-posed in the class of sets of finite perimeter and admits minimizers -- note that the $L^1$ convergence implies the convergence almost everywhere for subsequences, thus preserving the constraint $E\supset\Omega$ in the limit.
Notice, however, that uniqueness is in general false as show by the examples in \cite[\S~8.2]{ATW}. The approximate flow is, hence, defined as:
\[
E^{(h)}_t:=E^{(h)}_{\lfloor t\rfloor}\quad\forall\;t\geq0,
\]
where $\lfloor t\rfloor\in\N$ is the integer part of $t$, namely $\lfloor t \rfloor \leq t <\lfloor t\rfloor+1$.

\subsection{Regularity of approximate flows}
It follows from the regularity theory in geometric measure theory that the sets $E^{(h)}_{t}$ have $C^{1,1}$ boundaries. To see this, first we note that the functional $\cF(\cdot,h,E^{(h)}_{i})$ can be written in the following way:
\begin{align}\label{e.F2}
\cF(E,h,E^{(h)}_{i}) & =
\Pe(E)+\int_{\R^n}u_{i,h}(x)\;\chi_E(x)\;\dd x+\int_{\R^n}u_{i,h}(x)\;\chi_{E^{(h)}_{i}}(x)\;\dd x,
\end{align}
where we set $u_{i,h}:=h^{-1}d_i$ and $d_{i}$ the signed distance from \smash{$\de E^{(h)}_i$}:
\begin{equation}\label{e.di}
d_{i}(x):=
\begin{cases}
\dist\big(x,\de E^{(h)}_{i}\big) & \text{if }\;x\notin E^{(h)}_i,\\
-\dist\big(x,\de E^{(h)}_{i}\big) & \text{if }\;x\in E^{(h)}_i.
\end{cases}
\end{equation}
The last term in \eqref{e.F2} is a constant not depending on $E$. Therefore, it turns out that $E^{(h)}_{i+1}$ is also a minimizer of the functional $\cG(\cdot, h, E^{(h)}_{i})$:
\begin{equation}\label{e.G}
\cG(E, h, E^{(h)}_{i}) := \Pe(E)+\int_{\R^n}u_{i,h}(x)\,\chi_E(x)\,\dd x.
\end{equation}
Note that
\[
\cG\Big(E \cap (E^{(h)}_{i})^{co},h,E^{(h)}_i\Big) \leq \cG\big(E ,h,E^{(h)}_i\big), 
\]
with equality only if $E\subseteq (E^{(h)}_{i-1})^{co}$. Hence, it follows by a simple induction argument that
\begin{equation}\label{e.bounded}
E^{(h)}_{t}\subset (E_0)^{co}\quad \forall\;t\geq 0.
\end{equation}

In turns, this implies that the sets $E^{(h)}_{i}$ are uniform $\Lambda$-minimizers of the perimeter for $\Lambda=\sigma\,h^{-1}$, where $\sigma>0$ is a given constant independent of $h$.
Namely, there exists $R>0$ such that for all $i\in\N$, $x\in\R^n$ and $0<r<R$, it holds
\begin{equation}\label{e.L min prima}
\Pe(E^{(h)}_i,B_r(x))\leq \Pe(F,B_r(x))+\sigma\,h^{-1}\,r^n\quad\forall\;F\sdif E^{(h)}_i\Subset B_r(x).
\end{equation}
From the regularity theory of $\Lambda$-minimizers (see \cite{Alm}, \cite{Bom}), if $n\leq 7$, it follows that $\de E^{(h)}_i\in C^{1,1/2}$ and the following density estimates hold (see \cite[Proposition~3.4]{Tam}):
\begin{gather}
\frac{\omega_{n-1}}{n}-\sigma\,h^{-1}\,r\leq \frac{\min\big\{|E^{(h)}_i\cap B_r(x)|,|E^{(h)}_i\setminus B_r(x)|\big\}}{r^n} \qquad \forall \; x \in \R^n,\label{e.density2}\\
\omega_{n-1} - (n-1)\,\sigma\,h^{-1}\,r \leq \frac{\Pe(E_i^{(h)},B_r(x))}{r^{n-1}} \qquad \forall \; x\in \de E_{i}^{(h)}.
\label{e.density2bis}
\end{gather}
Moreover, since by the $C^{1,1/2}$ regularity we can always reduce to a classical nonparametric setting, from the regularity theory for the obstacle problem (see, for example, \cite{Ca, KS}) and $u_{i,h}$ Lipschitz, it follows that $\de E^{(h)}_i\in C^{1,1}$ -- more details are given in Appendix~\ref{a.almost min}.

\subsection{Uniform distance estimate}
The main analytical estimate exploited in the proof of Theorem~\ref{t.main} is the following on the distance between two successive boundaries of the approximate flow.

\begin{proposition}\label{p.time unif}
There exists a dimensional constant $\gamma(n)>0$, such that
\begin{equation}\label{e.unif dist}
\dist\big(\de E^{(h)}_{i+1}, \de E^{(h)}_{i}\big)\leq \gamma(n) \; \sqrt{h}
\quad \forall \ i\in\N, \ \forall \ h>0.
\end{equation}
\end{proposition}

The proof of Proposition~\ref{p.time unif} follows by a simple adaptation of the arguments in \cite{LuSt}. For readers' convenience, we give here a detailed proof.

We premise the following density estimate for one-sided minimizers of the perimeter. The estimate can be easily deduce from the original arguments by De Giorgi exploited for minimizers \cite{DG} (see also \cite{Gi}).

\begin{lemma}\label{l.one side}
There exists a dimensional constant $\theta=\theta(n)>0$ with this property.
Let $E\subset B_R\subset \R^n$ be a Caccioppoli set such that $0\in \de E$ and 
\begin{equation}\label{e.onesided}
\Pe(E,B_R) \leq \Pe (F,B_R) \quad \forall \; E\subseteq F, \; F\setminus E \Subset B_R.
\end{equation}
Then,
\begin{equation}\label{e.density3}
\theta\,r^n\leq |B_r\setminus E| \quad \forall \; 0<r<R.
\end{equation}
\end{lemma}

\begin{proof}
For $r<R$, set $F_r:=E\cup B_r$. Note that, for almost every $r>0$, it holds
\begin{gather*}
\Pe(F_r)=\cH^{n-1}(\de B_r\setminus E)+\Pe(E,\R^n\setminus B_r(x)),\\
\Pe(B_r\setminus E)=\cH^{n-1}(\de B_r\setminus E)+\Pe(E,B_r),\\
\Pe(E)=\Pe(E,B_r)+\Pe(E,\R^n\setminus B_r).
\end{gather*}
Indeed, if $E$ were smooth, these formulas follow for all the $r$ such that $B_r$ and $E$ have transversal intersections. Otherwise one can argue by approximation.

Using now \eqref{e.onesided}, we deduce that, for almost every $r>0$,
\begin{align}\label{e.half}
\Pe(F_r)&=\cH^{n-1}(\de B_r\setminus E)+\Pe(E,\R^n\setminus B_r(x))\notag\\
&\geq \Pe(E)\notag\\
&=\Pe(E,B_r)+\Pe(E,\R^n \setminus B_r)\notag\\
&=\Pe(B_r\setminus E)-\cH^{n-1}(\de B_r\setminus E)+\Pe(E,\R^n\setminus B_r(x)).
\end{align}
By the isoperimetric inequality \cite[Corollary~1.29]{Gi}, there exists a dimensional constant $C>0$, such that
\begin{equation}\label{e.dif ineq}
C\,|B_r\setminus E|^{\frac{n-1}{n}}\leq \Pe(B_r\setminus E) \stackrel{\eqref{e.half}}{\leq}  2\,\cH^{n-1}(\de B_r\setminus E).
\end{equation}

Setting $f(r):=|B_r\setminus E|$, by the coarea formula \cite[3.4.4]{EG}, it holds
\[
\cH^{n-1}(\de B_r\setminus E)=f'(r) \quad \text{for a.e. } r>0.
\]
Hence, \eqref{e.dif ineq} reads as
\[
f(r)^{\frac{n-1}{n}}\leq 2\,C^{-1}\, f'(r).
\]
Integrating \eqref{e.dif ineq} we get the desired \eqref{e.density3} for a dimensional constant $\theta>0$.
\end{proof}

Using Lemma~\ref{l.one side}, we can give a proof of the uniform bound in Proposition~\eqref{p.time unif}.

\begin{proof}[Proof of Proposition~\ref{p.time unif}]
We claim that \eqref{e.unif dist} holds for
\begin{equation}\label{e.gamma}
\gamma:=2\sqrt{\frac{n\,\omega_n}{\theta}}+1,
\end{equation}
where $\theta$ is the constant in \eqref{e.density3}.

Set for simplicity of notation $L_1:=E^{(h)}_{i+1}$ and $L_0:=E^{(h)}_{i}$ and assume by contradiction that there exists a point $x\in \de L_1\setminus L_0$ such that
\[
\dist(x,L_0) > \gamma\,\sqrt{h}.
\]
Let $r:=\gamma\,\sqrt{h}/2$ and note that, since $B_r(x)\cap L_0=\emptyset$, $L_1$ satisfies a one-sided minimizing property in $B_r(x)$.
Indeed, let $F$ be such that
\[
F\subset L_1 \quad \text{and} \quad |L_1\setminus F|\subset\subset B_r(x).
\]
From $\cG(L_1,h,L_0)\leq \cG(F,h,L_0)$ and $u_{i,h}\vert_{B_r(x)}>0$ (notation as in \eqref{e.F2}), it follows that
\begin{align*}
\Pe(L_1,B_r(x))
\leq \Pe(F, B_r(x)).
\end{align*}
This implies that we can apply Lemma~\ref{l.one side} to $B_r(x)\setminus L_1$ and, hence, the density estimate \eqref{e.density3} gives:
\begin{equation}\label{e.dens1}
|L_1\cap B_{r}(x)|\geq \theta\,\left(\frac{\gamma\,\sqrt{h}}{2}\right)^{n}.
\end{equation}

On the other hand, set $L_3:=L_1\setminus B_{r}(x)$. By the minimizing property
\[
\cG(L_1,h,L_0)\leq \cG(L_3,h,L_0)
\]
and $u_{i,h}\vert_{B_r(x)} \geq \gamma/(2\sqrt{h})$, we get easily the following reversed bound:
\begin{equation}\label{e.dens2}
\frac{\gamma\,|L_1\cap B_{r}(x)|}{2\,\sqrt{h}}\leq n\,\omega_n\,\left(\frac{\gamma\,\sqrt{h}}{2}\right)^{n-1}.
\end{equation}
Clearly, \eqref{e.dens1} and \eqref{e.dens2} imply $\gamma\leq 2\sqrt{n\,\omega_n/\theta}$, which contradicts \eqref{e.gamma}.

Similarly, in the case there exists $x\in \de L_1 \cap L_0$ with $\dist(x,\de L_0)> \gamma \sqrt{h}$, we argue in the same way, noticing that $L_1$ turns out to be one-sided minimizing in a neighborhood of $x$.
\end{proof}

\subsection{Weak flow with obstacle}
Though it is not needed to the proof of Theorem~\ref{t.main}, we note that Proposition~\ref{p.time unif} also leads to the existence of a limit flow with obstacle. Indeed, from the very definition of discrete flow, it follows easily that
\[
\Pe(E^{(h)}_t)\leq \Pe(E_0)\quad\text{for every}\quad h,t\geq0.
\]
Hence, recalling \eqref{e.bounded} and the compactness (2) \S~\ref{ss.caccioppoli}, by a diagonal argument we find a subsequence $h$ (not relabelled) and sets $E_t$ such that
\[
E^{(h)}_t\to E_t\quad\text{as }\, h\to 0\quad \forall \;0\leq t\in\Q.
\]
Moreover, using Proposition~\ref{p.time unif}, one can show that for the whole discrete flow a uniform H\"older continuity in time in the $L^1$ topology holds (the proof is postponed to Appendix~\ref{a.estimates}).

\begin{proposition}\label{p.time cont}
There exists a constant $C>0$ such that
\begin{equation}\label{e.time cont}
|E^{(h)}_t\sdif E^{(h)}_s|\leq C\,|s-t|^{\frac{1}{2}}\quad\forall\;h>0,\;\forall\;
t,s\geq h>0.
\end{equation}
\end{proposition}

Clearly, this allows us to pass into the limit for every $t\geq0$ and find a limit flow $E_t$ satisfying the continuity estimate:
\[
|E_t\sdif E_s|\leq C\,|s-t|^{\frac{1}{2}} \quad \forall  \
t,s>0.
\]

\section{Monotone flows with obstacle}\label{s.mc flow}
Since we are interested in the asymptotics of the evolution with obstacle, we can restrict ourself to the case of ``nested'' flows, i.e.~flows satisfying $E_t\subseteq E_s$ for every $0\leq s\leq t$.
For the smooth flow the right condition to look at is the mean-convexity of the initial set. In the context of Caccioppoli sets there are different ways to generalize this notion, such as the local pseudo-convexity introduced by Miranda \cite{Mi} or the minimizing hulls (also called subsolutions) considered in \cite{BaGoMa, BaTa, HuIl}.
All these notions are variants of the one-sided minimization property introduced in the previous section. For our purposes, the minimizing hulls considered by Huisken and Ilmanen \cite{HuIl} fulfil.

\begin{definition}\label{d.mc}
A set $E\subseteq\R^n$ is a \textit{Minimizing Hull} in $\cO\subseteq\R^n$ open if
\begin{equation}\label{e.mc}
\Pe(E,\cO)\leq \Pe(F,\cO)\quad\forall\;E\subseteq F\quad\text{such that}\quad F\setminus E\Subset \cO.
\end{equation}
\end{definition}

We often do not specify the open set when $\cO=\R^n$. It is easy to verify that a minimizing hull $E$ with smooth boundary is mean-convex, while the reverse implication is in general false. Simple consequences of Definition~\ref{d.mc} are the following two properties.
\begin{itemize}
\item[(1)] \textit{If $E\subseteq\R^n$ is a minimizing hull and $F\subseteq\R^n$, then
\begin{equation}\label{e.mc3}
\Pe(E\cap F)\leq \Pe(F).
\end{equation}
}
Indeed, from the minimizing hull property $\Pe(E)\leq \Pe(E\cup F)$ and from  \eqref{e.union int}, we have
\begin{equation*}
\Pe(E\cap F)\leq \Pe(E)+\Pe(F)-\Pe(E\cup F) \leq \Pe(F).
\end{equation*}
\item[(2)] \textit{If $\{E_k\}_{k\in\N}$ is a sequence of minimizing hulls and $\chi_{E_k}\to\chi_E$ in $L^1$, then $E$ is a minimizing hull.}
Indeed, given $E\subset F$ such that $F \setminus E\Subset \R^n$, by the minimizing hull property of $E_k$ we have
\begin{equation}\label{e.mc h}
\Pe(E_k)\leq \Pe(E_k\cup F).
\end{equation}
On the other hand, $E_k\cap F\to E\cap F=E$ and, by semicontinuity (1) \S~\ref{ss.caccioppoli},
\begin{align*}
\Pe(E)&\leq \liminf_{k\to+\infty}\Pe(E_k\cap F)\\
&\stackrel{\mathclap{\eqref{e.union int}}}{\leq}\;
\liminf_{k\to+\infty}\big[\Pe(E_k)+\Pe(F)-\Pe(E_k\cup F)\big]\\
&\stackrel{\mathclap{\eqref{e.mc h}}}{\leq}\;\Pe(F).
\end{align*}
\end{itemize}

\subsection{Maximal solutions}
Given a minimizing hull as initial set, it is possible to define uniquely a maximal approximate flow.
The main observation in this regard is contained in the following lemma.

\begin{lemma}\label{l.mc}
Let $E_0\subset\R^n$ be a bounded closed minimizing hull such that
\[
\Omega\subseteq E_0 \quad \text{and} \quad |\de E_0|=0.
\]
Then, the following holds:
\begin{itemize}
 \item[(i)] any minimizer $E\supset \Omega$ of $\cG(\cdot, h, E_0)$ is a minimizing hull and $E\subseteq E_0$;
 \item[(ii)] if $E'$ is any other minimizer, then $E\cup E'$ and $E\cap E'$ are minimizers of $\cG(\cdot, h, E_0)$ as well. 
\end{itemize}
\end{lemma}

\begin{proof}
Let $u_{0,h} = h^{-1}d_0 \in L^\infty(\R^n)$, with $d_0$ the rescaled signed distance from $\de E_0$ in \eqref{e.di}, and for simplicity let us write $\cG(\cdot)$ for $\cG(\cdot, h, E_0)$.
We start proving that $E\subseteq E_0$.
Indeed, note that
\begin{align*}
\cG(E)&\leq \cG(E\cap E_0)=\Pe(E\cap E_0)+\int_{\R^n} u_{0,h}\,\chi_{E\cap E_0}\\
&\stackrel{\mathclap{\eqref{e.mc3}}}{\leq} \ \Pe(E)+\int_{\R^n}u_{0,h}\,\chi_{E} - \int_{\R^n}u_{0,h}\,\chi_{E\setminus E_0}\\
&= \cG(E) -\int_{\R^n}u_{0,h}\,\chi_{E\setminus E_0}.
\end{align*}
Since $u_{0,h} > 0$ in $\R^n\setminus E_0$, this implies $E\subseteq E_0$ a.e.

Next we show that $E$ is a minimizing hull. Let $E\subseteq F$ and $F\setminus E\Subset\R^n$. From the minimizing property of $E$ we infer the following:
\begin{align}\label{e.bo}
\cG(E) & = \Pe(E)+\int_{\R^n} u_{0,h}\,\chi_{E}\notag\\
&\leq \cG(F \cap E_0)\notag \\
&= \Pe(F\cap E_0)+\int_{\R^n} u_{0,h}\,\chi_{F\cap E_0}\notag\\
&\stackrel{\eqref{e.mc3}}{\leq} \Pe(F)+\int_{\R^n} u_{0,h}\,\chi_{F\cap E_0}.
\end{align}
From $E\subseteq E_0$ and \eqref{e.bo} we have that
\[
 \Pe(E) \leq \Pe(F) + \int_{\R^n} u_{0,h}\,\chi_{(F \setminus E) \cap E_0} \leq \Pe(F),
\]
where we used $u_{0,h}\vert_{E_0}\leq0$. This shows that $E$ is a minimizing hull.

Finally, let $E'$ be another minimizer of $\cG$.
From the minimizing property of $E$, we get
\begin{align}
\cG(E)&\leq \cG(E\cap E')=\Pe(E\cap E')+\int_{\R^n}u_h\,\chi_{E\cap E'}, \label{e.min u-i 1}\\
\cG(E)&\leq \cG(E\cup E')=\Pe(E\cup E')+\int_{\R^n}u_h\,\chi_{E\cup E'}. \label{e.min u-i 2}
\end{align}
Summing the two inequalities, we get
\begin{align*}
 2\, \cG(E) & \leq \Pe(E\cap E')+\Pe(E\cup E')+\int_{\R^n}u_h\,\chi_{E\cap E'} + \int_{\R^n}u_h\,\chi_{E\cup E'}\\
& \stackrel{\mathclap{\eqref{e.union int}}}{\leq} \; \Pe(E)+\Pe(E')+\int_{\R^n}u_h\,\chi_{E} + \int_{\R^n}u_h\,\chi_{E'}\\
& =\cG(E)+\cG(E').
\end{align*}
Since $E'$ is a minimizer, i.e.~$\cG(E)=\cG(E')$, we deduce that \eqref{e.min u-i 1} and \eqref{e.min u-i 2} are equalities, thus concluding that $E \cap E'$ and $E\cup E'$ are both minimizers of $\cG$.
\end{proof}

A simple first corollary of Lemma~\ref{l.mc} is the existence of a maximal minimizer for $\cG$.

\begin{corollary}\label{c.mc}
 Let $E_0$ be as in Lemma~\ref{l.mc}. Then, there exist a maximal minimizer $E_{{\max}}$ of $\cG$ in the following sense: if $E$ is any other minimizer of $\cG$, then $E \subseteq E_{{\max}}$.
\end{corollary}

\begin{proof}
We define $E_{{\max}}$ as a minimizer which maximize the volume, i.e.
\begin{align}
 |E_{{\max}}| &= {\max}\big\{|E| \ : \ E \ \text{minimizer of} \ \cG\big\},\label{e.def1}
\end{align}
If $E$ is any other minimizer of $\cG$, from Lemma~\ref{l.mc} we deduce that $E\cup E_{{\max}}$ is also a minimizer. Hence, since
\[
|E_{{\max}}|\leq |E\cup E_{{\max}}|, 
\]
from \eqref{e.def1} we infer that $E \subseteq E_{{\max}}$.
\end{proof}

From now on, we will call the flow constructed from these special solutions the \textit{maximal} approximate flows.
Similarly, we deduce the following proposition from Lemma~\ref{l.mc}.

\begin{proposition}\label{p.mc flow}
Let $E_0\subseteq\R^n$ be a minimizing hull with
\[
\Omega\subset E_0 \quad \text{and} \quad |\de E_0|=0,
\]
and, for every $h>0$, let $E^{(h)}_{{\max}, t}$ denote the maximal flows.
Then, the following holds:
\begin{itemize}
\item[(i)] $E^{(h)}_{{\max},t}\subseteq E^{(h)}_{{\max},s}$ for every $0\leq s\leq t$;
\item[(ii)] $E^{(h)}_{{\max},t}$ is a minimizing hull for every $t\geq 0$.
\end{itemize}
\end{proposition}

\begin{proof}
The proof follows readily from the previous Lemma~\ref{l.mc}, noticing that, by the regularity of the minimizers, it holds $|\de E_i^{(h)}|=0$.
\end{proof}

\subsection{Monotonicity}
In the proof of Theorem~\ref{t.main} we need also the following refined monotonicity property. The proof exploits the same arguments used above.

\begin{lemma}\label{l.inclusion}
Let $E_0$ and $F_0$  be two closed bounded minimizing hulls such that
\[
\Omega\subseteq E_0\subseteq F_0 \quad\text{and}\quad |\de E_0|=|\de F_0|=0.
\]
Then, the maximal minimizers $E_{{\max}}$ of $\cG(\cdot, h,E_0)$ and $F_{{\max}}$ of $\cG(\cdot, h,F_0)$ satisfy 
\begin{equation}\label{e.monotone}
 E_{{\max}}\subseteq F_{{\max}}
\end{equation}
\end{lemma}

\begin{proof}
For simplicity, set $u_0:=h^{-1}d_{\de E_0}$ and $u_1:=h^{-1}d_{ \de F_0}$, where $d_{\de E_0}$ and $d_{\de F_0}$ are the signed distances from $\de E_0$ and $\de F_0$ respectively, as defined in \eqref{e.di}.
Using the minimizing properties, we get:
\begin{align}
\Pe(E_{\max})+\int_{\R^n}u_0\,\chi_{E_{\max}}&\leq \Pe(E_{\max}\cap F_{\max})+\int_{\R^n}u_0\,\chi_{E_{\max}\cap F_{\max}},\label{e.min0}\\
\Pe(F_{\max})+\int_{\R^n}u_1\,\chi_{F_{\max}}&\leq \Pe(E_{\max}\cup F_{\max})+\int_{\R^n}u_1\,\chi_{E_{\max}\cup F_{\max}}.\label{e.min1}
\end{align}
Summing these two inequalities, and using \eqref{e.union int}, we get
\[
\int_{\R^n}u_0\,\chi_{E_{\max}}+\int_{\R^n}u_1\,\chi_{F_{\max}}\leq 
\int_{\R^n}u_0\,\chi_{E_{\max}\cap F_{\max}}+\int_{\R^n}u_1\,\chi_{E_{\max}\cup F_{\max}},
\]
which in turn implies
\begin{equation}\label{e.abs}
\int_{\R^n}(u_0-u_1)\,\chi_{E_{\max}\setminus F_{\max}} \leq 0.
\end{equation}
Since $u_0 \geq u_1$ in $E_0$ and $E_{\max}\subseteq E_0$ by Lemma~\ref{l.mc}, we infer that \eqref{e.abs} is an inequality.
This implies that also \eqref{e.min0} and \eqref{e.min1} are equalities, i.e.~$E_{\max}\cup F_{\max}$ is a minimizer of $\cG(\cdot, h_1, F_0)$.
By maximality of the solution, we conclude \eqref{e.monotone}.
\end{proof}

\section{Least barrier}\label{s.mc hull}
In this section we prove Theorem~\ref{t.main}.
We show that the mean-convex hull $\Omega^{mc}$ can be characterized by the approximate asymptotic evolutions (well defined thanks to Proposition~\ref{p.mc flow} (i)):
\[
 E^{(h)}_{{\max},\infty} := \bigcap_{t \geq 0} E^{(h)}_{{\max},t}.
\]

To this aim, we start showing the regularity of such asymptotics.

\subsection{The asymptotic limit $E^{(h)}_{{\max},\infty}$}\label{ss.asympt}
From Proposition~\ref{p.mc flow} and (2) \S~\ref{s.mc flow}, it follows that $E^{(h)}_{{\max},\infty}$ is a minimizing hull for every $h>0$. In this section we prove that every $E_{{\max}, \infty}^{(h)}$ is stationary under the approximate mean curvature flow with obstacle and enjoys uniform regularity properties.

\begin{proposition}\label{p.asymptotic}
For every $0<h' \leq h$, $E_{{\max},\infty}^{(h)}$ is the maximal minimizer of $\cG(\cdot, h', E_{{\max},\infty}^{(h)})$.
In particular, $E^{(h)}_{{\max},\infty} \subseteq E^{(h')}_{{\max},\infty}$.
\end{proposition}

\begin{proof}
We start proving that $E_{{\max},\infty}^{(h)}$ is a minimizer of $\cG(\cdot, h, E_{{\max},\infty}^{(h)})$.
We proceed by contradiction. Assume there exists $F\subset E_{\max,\infty}^{(h)}$ such that
\begin{equation}\label{e.contra}
 \cG\big(F, h, E_{{\max},\infty}^{(h)}\big) < \cG\big(E_{{\max},\infty}^{(h)}, h, E_{{\max},\infty}^{(h)}\big).
\end{equation}
Note that, by the semicontinuity of the perimeter (1) \S~\ref{ss.caccioppoli} and the locally uniform convergence $d_i\to d_\infty$, where $d_\infty$ is the signed distance to $\de E_{\max,\infty}^{(h)}$ as in \eqref{e.di}, we have
\begin{gather}
 \cG\big(F, h, E_{{\max},\infty}^{(h)}\big) = \lim_{i\to+\infty} \cG\big(F, h, E_{{\max},i}^{(h)}\big),
 \label{e.cont}\\
 \cG\big(E_{{\max},\infty}^{(h)}, h, E_{{\max},\infty}^{(h)}\big) \leq \liminf_{i\to\infty} \cG \big(E_{{\max},i+1}^{(h)}, h, E_{{\max},i}^{(h)}\big).
\label{e.semicont}
\end{gather}
From \eqref{e.contra}, \eqref{e.cont} and \eqref{e.semicont}, we infer that, for $i$ big enough,
\[
 \cG\big(F, h, E_{{\max},i}^{(h)}\big) < \cG\big(E_{{\max},i+1}^{(h)}, h, E_{{\max},i}^{(h)}\big),
\]
thus contrasting with the minimizer property of $E_{{\max},i+1}^{(h)}$.

Now, note that 
\[
 \cG\big(E_{{\max},\infty}^{(h)}, h, E_{{\max},\infty}^{(h)}\big) \leq \cG\big(F, h, E_{{\max},\infty}^{(h)}\big),\quad\forall \; F \subseteq E_{{\max},\infty}^{(h)}
\]
implies that, for all $h'\leq h$, (recall that $d(\cdot, \de E^{(h)}_{\max,\infty}) \leq 0$ on $\de E^{(h)}_{\max, \infty}$)
\begin{align*}
 \Pe \big(E_{{\max},\infty}^{(h)}\big) & \leq \Pe(F) - \int_{E_{{\max},\infty}^{(h)} \setminus F}h^{-1} d(x, \de E_{{\max},\infty}^{(h)}) \\
 & \leq \Pe(F) - \int_{E_{{\max},\infty}^{(h)} \setminus F}h'^{-1} d(x, \de E_{{\max},\infty}^{(h)}),
\end{align*}
which, in turns, leads to the minimizing property for $\cG(\cdot, h',E^{(h)}_{\max, \infty})$:
\[
 \cG\big(E_{{\max},\infty}^{(h)}, h', E_{{\max},\infty}^{(h)}\big) \leq \cG\big(F, h', E_{{\max},\infty}^{(h)}\big).
\]
Finally, since $E_{{\max},\infty}^{(h)} \subseteq E_0$, the last assertion follows by induction from Lemma~\ref{l.inclusion}.
\end{proof}

In particular, recalling the regularity theory for almost minimizers of the perimeter (see also Appendix~\ref{a.almost min}), it follows from Proposition~\ref{p.asymptotic} that $E_{{\max},\infty}^{(h)}$ is $C^{1,1}$ regular and, moreover, the asymptotic approximate evolutions $E_{{\max},\infty}^{(h)}$ have a $L^1$-limit as $h\to 0$,
\[
 E_{{\max},\infty}^{(h)} \uparrow E_{{\max},\infty} := \bigcup_{h>0} E_{{\max},\infty}^{(h)}.
\]

In order to show regularity estimates for the limit $E_{{\max},\infty}$, we prove in the next proposition that uniform $C^{1,1}$ estimates (i.e.~independent of $h$) hold for the approximate asymptotics.
In the sequel, for any set $E\subset\R^n$ such that $\de E \in C^{1,1}$, $\|A_{\de E}\|_{L^\infty}$ denotes the length of the second fundamental form of the boundary.

\begin{proposition}\label{p.uniform}
There exists a dimensional constant $c_0=c_0(n)>0$ such that
\begin{equation}\label{e.uniform}
 \|A_{\de E^{(h)}_{{\max}, \infty}}\|_{L^\infty} \leq c_0\, \|A_{\de \Omega}\|_{L^\infty} \quad \forall \; h>0.
\end{equation}
\end{proposition}

\begin{proof}
We start noticing the following:
\begin{itemize}
 \item[(a)] without loss of generality, up to homotetically rescaling the obstacle $\Omega$, we can assume that $\|A_{\de \Omega}\|_{L^\infty} = 1$;
 \item[(b)] since each $E^{(h)}_{{\max},\infty}$ is a minimizing hull, then
\begin{equation}\label{e.unif per}
\Pe (E^{(h)}_{{\max},\infty}, B_r(p)) \leq n\,\omega_n\,r^{n-1} \quad \forall \; h,r>0;
\end{equation}
 \item[(c)] $M_h := \de E^{(h)}_\infty \setminus \Omega$ is a stable minimal hypersurface: this follows from the Euler--Lagrange equation for $\cG(\cdot,h, E^{(h)}_{{\max},\infty})$, i.e.
\[
 H_{M_h}(\cdot) = h^{-1} d (\cdot, \de E^{(h)}_{{\max},\infty}),
\]
and the one-sided area minimizing property of $M_h$.
\end{itemize}

The proof of \eqref{e.uniform} is made by contradiction via a blow-up argument. Assume there exist a sequence $h_k\to0$ and points $p_k\in \de E^{(h_k)}_{{\max},\infty}$ such that
\begin{equation}\label{e.absurd}
 \alpha_k := 2\,|A_{\de E^{(h_k)}_{{\max},\infty}}(p_k)| \geq \|A_{\de E^{(h_k)}_{{\max},\infty}}\|_{L^\infty} \to +\infty.
\end{equation}
Set $r_k := \alpha_k^{-1}$ and consider the translated and rescaled sets
\[
F_k := r_k^{-1}\,\big(E^{(h_k)}_{{\max},\infty} - p_k\big).
\]
Note that $0\in \de F_k$ and $\de F_k \in C^{1,1}$ with
\begin{equation}\label{e.curv}
 \|A_{\de F_k}\|_{L^\infty}\leq 1 \quad \text{and} \quad |A_{\de F_k}(0)| = \frac{1}{2}.
\end{equation}
By the uniform bound on the perimeters \eqref{e.unif per}, it holds
\begin{equation}\label{e.unif per2}
\Pe (F_k, B_R) \leq n\,\omega_n\,R^{n-1} \quad \forall \; R>0.
\end{equation}
Hence, by the compactness (1) \S~\ref{ss.caccioppoli} and \eqref{e.curv}, up to extracting a subsequence (here and in the sequel not relabelled), we can infer that $F_k$ converge locally to a set $F$ such that $0 \in \de F$ and $F\in C^{1,1}$. Moreover, since limits of minimizing hulls, by (2) \S~\ref{s.mc flow} also $F$ is a minimizing hull.

The contradiction is now reached as follows. If there exists a subsequence such that
\[
 r_k^{-1}\dist(p_k,\Omega)\to \infty,
\]
then, by the stability of $M_h$, $\de F$ is a stable minimal hypersurface in $\R^n$.
Since $n\leq 7$ and \eqref{e.unif per2} holds, by the Schoen--Simon curvature estimates \cite{SS} it follows that $F$ is a half space, thus contradicting \eqref{e.curv}.

On the other hand, if
\[
\sup _k r_k^{-1}\dist(p_k,\Omega)<\infty,
\]
then, from $\|A_{\de \Omega}\|_{L^\infty} = 1$ we infer that, up to extracting a subsequence, the rescaled obstacles
\[
\Omega_k := r_k^{-1}\big(\Omega -p_k\big)
\]
converge locally to a closed half space $H$.
If $\de F \cap H =\emptyset$, then we can argue as above and deduce that $F$ needs to be itself a half space, contradicting \eqref{e.curv}.

If there exists $p \in \de F \cap H$, by the $C^{1,1}$ regularity of $\de F$, one can find $r>0$ such that $B_r(p) \cap \de F$ is a graph over $\de \Omega$.
Since $F$ is a minimizing hull, $\de F \cap B_r(p)$ is a supersolution of the minimal surface equation.
Therefore, by the strong maximum principle $\de F$ coincides with $\de H$ in $B_r$ and, by a unique continuation argument, $F = H$, again contradicting \eqref{e.curv}.
\end{proof}

As a straightforward corollary of the above proposition, we have the following.

\begin{corollary}\label{c.reg asympt}
 Let $\Omega \subset \R^n$ be a closed $C^{1,1}$ set and $E_0\supset \Omega$ a closed minimizing hull with $|\de E_0|=0$.
Then, the following holds:
\begin{itemize}
 \item[(i)] $E_{{\max},\infty}$ is minimizing hull;
 \item[(ii)] $\de E_{{\max},\infty} \in C^{1,1}$ with uniform estimated
\[
\|A_{\de E^{}_{{\max}, \infty}}\|_{L^\infty} \leq c_0\, \|A_{\de \Omega}\|_{L^\infty};
\]
\item[(iii)] $\de E_{{\max},\infty} \setminus \Omega$ is a smooth minimal hypersurface.
\end{itemize}
\end{corollary}

\subsection{Mean-convex hull}
Now we are ready for the proof of Theorem~\ref{t.main}.
The proof is made in several steps and the strategy is as follows: we construct a $C^{1,1}$ regular set containing $\Omega$ and show that it is actually the minimal barrier.

\subsubsection{Step 1}
Consider the closed $\eps$-neighbourhood of the obstacle $\Omega$:
\[
\Omega_\eps:=\big\{x:\dist(x,\Omega)\leq \eps\big\}.
\]
Note that $\Omega_\eps \downarrow \Omega$, i.e.
\[
\Omega_{\eps_1}\subseteq\Omega_{\eps_2}\quad\forall \; 0\leq\eps_1\leq\eps_2
\quad\text{and}\quad
\bigcap_{\eps>0}\Omega_\eps=\Omega_0.
\]
Moreover, by the $C^{1,1}$ regularity of $\de \Omega$ there exists $\eps_0>0$ such that $\de\Omega_\eps\in C^{1,1}$.
Let now $E_0$ be a closed convex set such that $\Omega_\eps\Subset\textup{int}(E_0)$ for every $\eps<\eps_0$ and let $E^{\eps}_{\max,\infty}$ be the asymptotic limit of the maximal flows starting at $E_0$ with respect to the obstacle $\Omega_\eps$.
Set 
\[
 \cE(\Omega) := \bigcap_{\eps>0} E^\eps_{\max,\infty}.
\]
We will show that $\Omega^{mc} = \cE(\Omega)$.

\subsubsection{Step 2}
The main ingredient for the proof of Theorem~\ref{t.main} is contained in the following proposition.

\begin{proposition}\label{p.max}
Let $\Omega$ and $E_0$ be as in Step 1.
Then, every minimal hypersurface $\Sigma$ with $\de \Sigma\subseteq \Omega$ is contained in $\cE(\Omega)$.
\end{proposition}

\begin{remark}
Note that, in view of the counterexamples in \S~\ref{s.mc barrier}, it is essential that $E_0$ is not just a generic minimizing hull containing the obstacle $\Omega$.
\end{remark}

\begin{proof}
Let $E^{(h),\eps}_{{\max},t}$ denote the approximate maximal flows starting at $E_0$ with respect to the obstacle $\Omega_\eps$.
We show that, for every minimal hypersurface $\Sigma$ with $\de \Sigma\subset \Omega$, it holds $\Sigma \subset E_{{\max},\infty}^{(h),\eps}$ for every $\eps>0$ and $h < \eps^2/(4 \gamma^2)$,
where $\gamma$ is the constant in Proposition~\ref{p.time unif}. 
This implies that
\[
 \Sigma \subset \bigcup_{h>0}E_{{\max},\infty}^{(h),\eps} = E^\eps_{\max,\infty}\quad \forall\; \eps>0,
\]
thus proving the proposition.

The proof of the claim is by contradiction. Assume there exists $i\in \N$ such that
\begin{equation}\label{e.abs3}
\Sigma\subset E^{(h)}_{{\max}, i} \quad\text{and}\quad \Sigma\setminus E^{(h)}_{{\max}, i+1}\neq \emptyset.
\end{equation}
Note that here we used the convex hull property for minimal surfaces which implies $\Sigma \subset E_0$.
Set for simplicity of notation $L:=E^{(h)}_{{\max}, i}$ and consider the closed set of points of minimum distance between $\de L$ and $\bar\Sigma$:
\[
W:=\big\{x\in\bar\Sigma\,:\,\dist(x,\de L)=\dist\big(\Sigma,\de L\big)\big\}.
\]
From Proposition~\ref{p.time unif} and \eqref{e.abs3}, we deduce that
\begin{equation*}
\dist(\Sigma,\de L)\leq \gamma\,\sqrt{h}.
\end{equation*}
Hence, since $2\,\gamma\,\sqrt{h}<\eps$ and $\de \Sigma\subset\Omega$ is distant at least $\eps$ from $\Omega_\eps$, the minimum distance is reached in the interior of $\Sigma$, i.e.~$W\subset\Sigma$.
Let $x_0\in W$ be a boundary point of $W\subset\Sigma$ for the induced topology, i.e.
\begin{equation}\label{e.bdy}
B_r(x_0)\cap(\Sigma\setminus W)\neq\emptyset\quad\forall\;r>0,
\end{equation}
and let $y_0\in\de L$ be such that $\dist(\Sigma,W)=|x_0-y_0|$.
Consider
\[
\Sigma'=\Sigma+y_0-x_0.
\]
We have that $\Sigma'\subset L$ and $\Sigma'\cap \de L\neq \emptyset$.
We can apply the classical strict maximum principle for the minimal surface equation and conclude that $\Sigma'\equiv \de L$ in a neighborhood of $x_0$, against \eqref{e.bdy}.
\end{proof}

\subsubsection{Step 3}
Next we notice that $\cE(\Omega)$ satisfies the regularity conclusion of Theorem~\ref{t.main}. Indeed, by the uniform estimate in Corollary~\ref{c.reg asympt} (ii), it follows that
\begin{equation}\label{e.c1}
\de \cE(\Omega) \in C^{1,1}. 
\end{equation}
Moreover, again appealing to the uniform estimates of the corollary, we have that $\de \cE(\Omega) \setminus \Omega$ is locally the limit of $\de E^\eps_{\max,\infty} \setminus \Omega_\eps$. Hence, from Corollary~\ref{c.reg asympt} (iii) we deduce that $\de\cE(\Omega)\setminus \Omega$ is a minimal hypersurface with boundary on $\de\Omega$.

\subsubsection{Step 4}
Next we show that $\cE(\Omega)$ is actually a global barrier.

\begin{proposition}\label{p.global}
Let $\Omega$ and $E_0$ be as in Step 1.
Then, $\cE(\Omega)$ is a global barrier,i.e.
\[
\Sigma \;\text{minimal hypersurface, }\;\de\Sigma\subset \cE(\Omega)\quad\Longrightarrow\quad \Sigma\subset\cE(\Omega).
\]
\end{proposition}

\begin{proof}
By Proposition~\ref{p.max}, it is enough to show that
\begin{equation}\label{e.hull}
\cE(\cE(\Omega)) = \cE(\Omega).
\end{equation}

To this aim, set for simplicity $\cE_1 := \cE(\Omega)$, $\cE_2 := \cE(\cE(\Omega))$ and $M:= \de \cE_2 \setminus \cE_1$.
We claim that
\begin{equation}\label{e.hull mc}
\de M \subset \Omega.
\end{equation}
Assume, indeed, there exists $x_0\in \de M \setminus \Omega$.
Then, in particular, since $\de M\subset \de \cE_1$, we have that $x\in \de \cE_1 \setminus \Omega$.
Then, by the regularity of $\cE$ in Step 3, there exists $0<r<\dist(x_0,\de \Omega)$ with these properties:
\begin{itemize}
 \item[(a)] $\Sigma_1:=B_r(x_0)\cap \de \cE_2$ and $\Sigma_2 := B_r(x_0)\cap \de \cE_1$ are graphs of functions $f_1, f_2: T \to T^\perp$, where $T$ is the tangent plane to $\de \cE_1$ at $x_0$;
\item[(b)] $f_1$ is a supersolution of the minimal surface equation, and $f_2$ a solution;
\item[(c)] $f_2 \leq f_1$ and $f_1(x_0)=f_2(x_0)$.
\end{itemize}
%
By the strong maximum principle for the minimal surface equation, $f_1\equiv f_2$, thus implying that
\[
B_r(x_0)\cap \de \cE_2 = B_r(x_0)\cap \de \cE_1.
\]
This contradicts $x_0\in \de M = \de \cE_2 \setminus \cE_1$.

The conclusion of the proof is now straightforward.
Since by Proposition~\ref{p.max} $\cE(\Omega)$ is a barrier for minimal hypersurfaces with boundary in $\Omega$, from \eqref{e.hull mc} it follows that $\de \cE(\cE(\Omega)) \subset \cE(\Omega)$, which together with the obvious inclusion $\cE(\Omega)\subseteq\cE(\cE(\Omega))$ gives \eqref{e.hull}.
\end{proof}

\subsubsection{Step $5$}
The proof of Theorem~\ref{t.main} now follows straightforwardly. By the previous steps, we deduce that $\cE(\Omega)$ is a global barrier containing $\Omega$ and satisfying the regularity conclusion of the theorem.

We need only to show that $\cE(\Omega)$ is the least possible barrier. To this aim, note that, since $\de \cE(\Omega) \setminus \Omega$ is a minimal surface with boundary in $\Omega$, then necessarily
\begin{equation}\label{e.bdry}
 \de \cE(\Omega) \subset \Omega^{mc}.
\end{equation}
The conclusion then follows noting that \eqref{e.bdry} implies $\cE(\Omega) \subset \Omega^{mc}$, because $\cE(\Omega)$ can be realized as the union of minimal hypersurfaces with boundary on $\de\cE(\Omega)$ (which then necessarily are contained in $\Omega^{mc}$), e.g.
\[
 \cE(\Omega) = \bigcup_{t\in\R} \Big(\cE(\Omega) \cap \big\{x\,:\,x_n=t\big\}\Big).
\]

\appendix

\section{Existence of a mean curvature flow with obstacle}\label{a.estimates}
Here we give the proof of the continuity estimate in Proposition~\ref{p.time cont} (restated below) leading to the existence of a weak mean curvature flow with obstacle.
The proofs we propose are simple adaptation of the ones for the weak flow without obstacle.
In particular, we continue following the arguments in \cite{LuSt}, where several estimates are simplified with respect to the ones in \cite{ATW}.

\begin{proposition}\label{p.time cont appendix}
There exists a constant $C>0$ such that
\begin{equation}\label{e.time cont app}
|E^{(h)}_t\sdif E^{(h)}_s|\leq C\,|s-t|^{\frac{1}{2}}\quad\forall\;h>0,\;\forall\;
t,s\geq h>0.
\end{equation}
\end{proposition}

\begin{proof}
Consider $\alpha<\beta\,h^{-1/2}$, where $\beta>0$ is a dimensional constant to be fixed momentarily.
Let $l\in\N\setminus\{0\}$ and start estimating $|E^{(h)}_{l+1}\setminus E^{(h)}_l|$.
Set
\begin{align*}
I_1:= \left\{x\in E^{(h)}_{l+1}\setminus E^{(h)}_l\,:\,\dist(x,\de E^{(h)}_l)\geq \alpha\,h\right\}
\quad\text{and}\quad I_2:=\left(E^{(h)}_{l+1}\setminus E^{(h)}_l\right)\setminus I_1.
\end{align*}
The estimate of $I_1$ is straightforward: using
\[
\cF\big(E^{(h)}_{l+1},h,E^{(h)}_{l}\big)\leq \cF\big(E^{(h)}_{l},h,E^{(h)}_{l}\big)=\Pe(E^{(h)}_l),
\]
we infer that
\begin{align}
|I_1|&\leq (\alpha\,h)^{-1}\int_{E^{(h)}_l\sdif E^{(h)}_{l+1}}\dist(x,\de E^{(h)}_l)\notag\\
&\leq \alpha^{-1}\left(\Pe(E^{(h)}_l)-\Pe(E^{(h)}_{l+1})\right).\label{e.I1}
\end{align}
For what concerns $I_2$, we note that
\[
I_2\subset \bigcup_{x\in\de E^{(h)}_l}B_{2\,\alpha\,h}(x).
\]
Hence, by Besicovitch's Covering Theorem (cp.~\cite[Theorem~2.7]{Mattila}), we can find $\xi(n)$ family of countable, disjoint balls covering $I_2$, where $\xi(n)$ is the Besicovitch dimensional constant.
Next, note that, by Proposition~\ref{p.time unif} we have that 
\[
u_{l,h}(x) = h^{-1}\dist(x,\de E^{(h)}_l)\leq h^{-1}\,\gamma\,\sqrt{h}<\gamma\,h^{-1/2}\quad\forall\;x\in I_2.
\]
This implies that $E^{(h)}_{l+1}$ is a $\Lambda$-minimizer of the perimeter, with $\Lambda=C\,\gamma\,h^{-1/2}$.
Hence, if $\beta$ is sufficiently small to have
\[
\Lambda\,2\,\alpha\,h < 2\,C\,\gamma\,\beta < \omega_{n-1}/2,
\]
we can apply the density estimate \eqref{e.density2bis} and infer
\begin{align*}
\big|\big(E^{(h)}_{l+1}\setminus E^{(h)}_l\big)\cap B_{2\,\alpha\,h}(x)\big| \leq C\, (\alpha\,h)^n \leq C\,\alpha \,h\,\Pe(E^{(h)}_{l}, B_{2\alpha h}).
\end{align*}
Therefore, considering the local finiteness of the covering, we finally get:
\begin{equation}\label{e.I2}
|I_2|\leq C\,\alpha \,h\,\Pe(E^{(h)}_l).
\end{equation}
Summing \eqref{e.I1} and \eqref{e.I2}, we conclude
\[
|E^{(h)}_{l+1}\setminus E^{(h)}_l|\leq \alpha^{-1}\left(\Pe(E^{(h)}_l)-\Pe(E^{(h)}_{l+1})\right)+
C\,\alpha \,h\,\Pe(E^{(h)}_l).
\]

Now consider the case $s=i\,h$ and $t=j\,h$ for $i,j\in\N$, $0<i<j$.
Then, by triangular inequality,
\begin{align}\label{e.tri}
|E^{(h)}_{j}\setminus E^{(h)}_i|&\leq\sum_{l=i}^{j-1}|E^{(h)}_{l+1}\setminus E^{(h)}_l|\notag\\
&\leq \sum_{l=i}^{j-1}\alpha^{-1}\left(\Pe(E^{(h)}_l)-\Pe(E^{(h)}_{l+1})\right)+
C\,\alpha \,h\,\Pe(E^{(h)}_l)\notag\\
&\leq \alpha^{-1}\left(\Pe(E^{(h)}_i)-\Pe(E^{(h)}_{j})\right)+
C\,\alpha \,h\,|j-i|\notag\\
&\leq C\,\alpha^{-1}+ C\,\alpha \,|t-s|.
\end{align}
Choosing $\alpha=\beta(n)\,|t-s|^{-1/2}$ and noting that $\alpha<\beta(n)\,h^{-1/2}$, we infer that
\begin{align}\label{e.cont1}
|E^{(h)}_{j}\setminus E^{(h)}_i|\leq C\,\sqrt{|t-s|}.
\end{align}
Clearly, by the piecewise definition of the approximating flow, it is enough to infer
\[
|E^{(h)}_{t}\setminus E^{(h)}_s|\leq C\,\sqrt{|t-s|}\quad\forall\;0<h<s<t.
\]
Since the estimate
\[
|E^{(h)}_{s}\setminus E^{(h)}_t|\leq C\,\sqrt{|t-s|}\quad\forall\;0<h<s<t,
\]
can be obtained analogously, this gives concludes the proof.
\end{proof}

\section{Regularity of the flow with obstacle}\label{a.almost min}
Here we recall the main arguments in order to infer the partial regularity of the approximate flow with obstacle.

The starting point is the following regularity result, due to Almgren \cite{Alm}.

\begin{theorem}\label{t.alm}
Let $E\subset\R^n$ be a $\Lambda$-minimizer of the perimeter at scale $R$, i.e.
\begin{equation}\label{e.L min}
P(E,B_r(x)) \leq P(F,B_r(x)) + \Lambda \,r^n\quad\forall\;x\in\R^n,\;\forall\;0<r<R.
\end{equation}
Then, there exists a set $\Sigma$ of Hausdorff dimension at most $n-8$ (empty if $n<8$ and discrete if $n=8$) such that $\de E\setminus \Sigma\in C^{1,1/2}$.
\end{theorem}

Note that, reversely, if $\de E\in C^{1,1/2}$, then $E$ is a $\Lambda$-minimizers for $\Lambda=1$ and $R>0$ accordingly chosen.

Showing that the approximate flow $E^{(h)}_i$ is made by almost minimizers is a standard computation, which we report for completeness.

\begin{lemma}\label{l.quasi}
There exist constant $C,R>0$ such that, for every $i\geq 1$, the sets $E^{(h)}_i$ are $(C\,h^{-1})$-minimizers of the perimeter at scale $R$.
\end{lemma}

\begin{proof}
As observed above, the minimizers of $\cG(\cdot,h,E^{(h)}_{i-1})$ are contained in $\overline{co}(E_0))$, so that
\begin{equation}\label{e.uh}
\|u_{i-1,h}\|_{L^\infty(\overline{co}(E_0))}\leq h^{-1}\overline{co}(E_0))
\quad\text{and}\quad \|\nabla u_{i-1,h}\|_{L^\infty}\leq h^{-1}.
\end{equation}

It is now very easy to show the $\Lambda$-minimizing property. Set for simplicity $\cG(\cdot) = \cG(\cdot,h,E^{(h)}_{i-1})$.
Let $R>0$ to be fixed momentarily and $x\in\de E$, where $E$ is a generic minimizer of $\cG$.
Consider $F$ a set such that $F\sdif E\subset\subset B_r(x)$, $0<r<R$.
In general, $F$ does not contain the obstacle $\Omega$.
Nevertheless, we can use $F\cup\Omega$ as a competitor:
\begin{align*}
\Pe(E)+\int_{\R^n}u_{i-1,h}\,\chi_E & = \cG(E) \leq\cG(F\cup\Omega)\\
&=\Pe(F\cup\Omega)+\int_{\R^n}u_{i-1,h}\,\chi_{F\cup\Omega}\notag\\
&\stackrel{\eqref{e.union int}}{\leq} \Pe(F)+\Pe(\Omega)-\Pe(F\cap \Omega)+\int_{\R^n}u_{i-1,h}\,\chi_{F\cup\Omega}.
\end{align*}
In turns, this implies
\begin{equation}\label{e.min}
\Pe(E)-\Pe(F)\leq \Pe(\Omega)-\Pe(F\cap\Omega)+\int_{\R^n}u_{i-1,h}\,\big(\chi_{F\cup\Omega}-\chi_E\big).
\end{equation}
Now, note that the integral term in the right hand side of \eqref{e.min} is simply estimated by
\[
\int_{\R^n}u_h\,\big(\chi_{F\cup\Omega}-\chi_F\big)\leq \|u_{i-1,h}\|_{L^\infty(F\cup E)}\,|F\sdif E|\stackrel{\eqref{e.uh}}{\leq} C\,h^{-1}\,r^n.
\]
For what concerns the first term in \eqref{e.min}, since $\de\Omega$ is $C^{1,1}$, it follows easily that $\Omega$ is a $1$-minimizer if $R$ is chosen sufficiently small, i.e., for every $x\in \R^n$ and $0<r < R$,
\[
P(\Omega,B_r(x)) \leq P(F,B_r(x)) + r^n \quad \forall \; \Omega\subset F, \quad F\setminus\Omega\Subset B_r(x).
\]
Therefore, we conclude from \eqref{e.min} the $\Lambda$-minimizing property of $E$ with $\Lambda=C\,h^{-1}$.
\end{proof}

By Theorem~\ref{t.alm}, $E^{(h)}_i$ has $C^{1,1/2}$ regular boundary up to a singular set $\Sigma$ of dimension at most $n-8$.
In fact, given the particular nature of the minimizers $E^{(h)}_i$, we can easily show that $\de E^{(h)}_i\setminus\Sigma$ is $C^{1,1}$, thus giving the optimal regularity for obstacle problem.

\begin{lemma}\label{l.reg}
For every $i\geq 1$, the sets $E^{(h)}_i$ are $C^{1,1}$ regular, up to possible singular set $\Sigma\subset\de E$ with Hausdorff dimension at most $n-8$ (countable if $n=8$).
\end{lemma}

\begin{proof}
Also the proof of this lemma follows from by now well-known arguments.
In particular, for points away from the obstacle we can use the the first variation of $\cG(\cdot, h, E^{(h)}_{i-1})$ to infer regularity.
Indeed, in a neighborhood of a regular point $x_0\in\de E^{(h)}_i$, we can parametrize $\de E^{(h)}_i$ by a function $\ph$ which satisfies the Euler--Lagrange equation
\begin{equation*}
\div\left(\frac{\nabla \ph}{\sqrt{1+|\nabla\ph|^2}}\right)=f,
\end{equation*}
with $f(y)=u_{i-1,h}(y,\ph(y))$.
Since $u_{i-1,h}\in W^{1,\infty}$ and $\ph\in C^{1,1/2}$, by well-known elliptic regularity theory, we infer the $C^{3,1/2}$ regularity of $\ph$, and hence of $\de E^{(h)}_i$.

On the other end, in a neighborhood of a regular point $x_0\in\de E^{(h)}_i\cap\Omega$, we can parametrize $\de E^{(h)}_i$ by a function $\ph$ which solves the non-parametric obstacle problem for the area functional:
\begin{equation}
\min \left\{\int_{D} \sqrt{1+|\nabla\zeta|^2}-f\,\zeta\;:\;\zeta\vert_{\de D}=\ph\vert_{\de D},\quad\zeta\geq\psi \right\},
\end{equation}
where $D\subset\R^{n-1}$ is a given smooth domain and $\psi$ is the parametrization of $\de \Omega$ in a neighborhood of $x_0$.
Following the theory in Kinderlehrer--Stampacchia \cite{KS}, we deduce the $C^{1,1}$ regularity of $\ph$.
(In passing, we note that every point $x_0\in \de E^{(h)}_i\cap\Omega$ is a regular point, since every tangent cone to $E^{(h)}_i$ in $x_0$ needs to be contained in a half space -- see Miranda [M]).
\end{proof}

\bibliographystyle{abbrv}
\bibliography{ref-barrier}
\end{document}